\documentclass[A4,11pt]{article}

\usepackage{t1enc}
\usepackage{amsmath,amstext,amscd,amsfonts,amsthm}
\usepackage{amssymb}
\usepackage{geometry}
\usepackage{enumerate}
\usepackage{tikz}
\usepackage{centernot,cancel}

\newtheorem{teorema}{Theorem}

\newtheorem{definicion}{Definition}

\newtheorem{lema}{Lemma}

\newtheorem{proposicion}{Proposition}
\theoremstyle{remark}

\newtheorem{observacion}{Observation}

\newcommand{\denselist}{\topsep 0pt\itemsep 0pt}

\newcommand{\tup}[1]{\langle #1 \rangle}

\newcommand{\col}{\mathbf c}

\newcommand{\set}[1]{\left\{ #1 \right\}}
\newcommand{\setdef}[2]{\set{ #1 \, : \, #2}}

\newcommand{\ident}{\!\sim\!}

\makeindex

\title{A note on Minimal Senders}

\author{Nerio Borges\\ Yachay Tech\\
	School of Mathematical Sciences\\
    and Information Technology\\
		nborges@yachaytech.edu.ec}

\date{\today}

\begin{document}

\maketitle

\begin{abstract}
	In this paper we prove that if a pair of graphs $G,H$ have senders, then they
	 necessarily have connected minimal senders; 
	 we also prove that given two fixed graphs that are either 3-connected or triangles
	 there are minimal $(G,H)$-senders with arbitrarily distant signal edges
	 and $(G,H)$-minimal graphs with arbitrarily large cycles,
	 thus showing there is no upper bound for the diameters of $(G,H)$-minimal graphs.
\end{abstract}

\section{Introduction}

\subsection{The Arrowing Relation, minimal graphs and senders}

Given a (simple, finite) graph $F=\tup{V_F, E_F}$ a {\em $2$-coloring of the edges of $F$}
(or simply an {\em edge coloring} of $F$)
is a function
\begin{equation}
\col: E_F\longrightarrow\set{0,1}
\end{equation}
We informally talk about {\em red} and {\em blue} colors instead of $0$ and $1$.

If $G$ and $H$ are two fixed graphs, an edge coloring of $F$ is {\em $(G,H)$-good}
(or simply a {\em good coloring} if $G,H$ are clear from context)
if there is no isomorphic copy of $G$ completely contained in the preimage of $0$
and no isomorphic copy of $H$ completely contained in the preimage of $1$ i.e.
there is no red $G$ and no blue $H$.

If $F$ has good colorings, we write
\begin{equation}
F\centernot\longrightarrow (G,H)
\end{equation}
and if $F$ has no good coloring, we write
\begin{equation}
F\longrightarrow (G,H)
\end{equation} 
This Arrowing relation has been widely studied \cite{Burr, Burr-MoR, lange2012use, Rodl-Sieggers, schelp2012some}
due to its connection with graph Ramsey theory \cite{diestel2005graph, graham1990ramsey}
which is a very active research area.
The computational complexity of its related decision problems has been also
studied \cite{Burr-MoR, Rutenburg, Schaefer} including some descriptive aspects.

There is also a {\em strong} Arrowing relation denoted by $\rightarrowtail$.
We write
\begin{equation}
F\rightarrowtail (G,H)
\end{equation}
if for every coloring of the edges of $F$ there is either an induced red subgraph isomorphic to $G$
or an induced blue subgraph isomorphic to $H$.
If there is a coloring with no induced red subgraph isomorphic to $G$
and no blue subgraph isomorphic to $H$ we write
\begin{equation}
F\centernot\rightarrowtail (G,H)
\end{equation}


A graph $F$ is {\em $(G,H)$-minimal} if
$F\longrightarrow (G,H)$ but $F'\centernot\longrightarrow (G,H)$ 
for every proper subgraph $F'\subset F$.
The set of all $(G,H)$-minimal graphs is denoted by $\mathcal R(G,H)$.
The set of all $(G,H)$-minimal graphs with respect to the
strong Arrowing is denoted by $\mathcal R^*(G,H)$.



A {\em positive $(G,H,e,f)$-sender} is a graph $\mathcal T$ containing two edges $e,f$ such that:
\begin{enumerate}
	\item\label{it:GoodCol} $\mathcal T$ has good colorings,
	\item\label{it:samecolor} $e$ and $f$ have the same color in every good coloring.
	\item\label{it:AltCol} There is a good coloring in which $e$ is red and one in which $e$ is blue.
\end{enumerate} 

A {\em negative $(G,H,e,f)$-sender} is a graph which satisfies conditions \ref{it:GoodCol} to \ref{it:AltCol}
but with the word `different' instead of the expression `the same' in condition \ref{it:samecolor}.

$\Gamma_3$ is the class consisting of all 3-connected graphs and the triangle $K_3$.
An important result related with this class is:

\begin{proposicion}{\cite{Burr}}\label{prop:BurrNegativeSenders}
	If $G,H\in \Gamma_3$ there exist positive and negative $(G,H)$-senders with arbitrarily distant signal edges.
\end{proposicion}

From now on $G$ and $H$ are two fixed graphs in $\Gamma_3$ unless otherwise explicitly stated.

\subsection{Contributions and organization of this paper}

There are three contributions presented in this paper. 
We prove that:
\begin{enumerate}\denselist
\item If $G,H$ have senders, then necessarily they have connected minimal senders (Theorem \ref{teo:MinimalSendersAreConnected}).
\item For any pair of fixed graphs $G,H$ in $\Gamma_3$ there are minimal senders with arbitrarily distant signal edges (Theorem \ref{teo:ArbitraryDistanceBetweenSignalEdges}).
\item For any pair of fixed graphs $G,H$ in $\Gamma_3$ there are elements of $\mathcal R(G,H)$ with arbitrarily large cycles (Theorem \ref{Teo:CicloNoAcotado}).
\end{enumerate}

Theorems \ref{teo:MinimalSendersAreConnected} and \ref{teo:ArbitraryDistanceBetweenSignalEdges}
are proved in Section \ref{sec:MinimalSenders}, Theorem \ref{Teo:CicloNoAcotado} is proved in Section \ref{MinimalSendersWithArbitrarilyLargeCycles} and Section \ref{sec:FurtherResearch} is dedicated
to present some comments on the results and some questions that are left open in this work.

To prove Theorem \ref{teo:ArbitraryDistanceBetweenSignalEdges} we use Proposition \ref{prop:BurrNegativeSenders}
and the fact that the usual order on the natural numbers is a well order. To prove Theorem \ref{Teo:CicloNoAcotado}
we use Theorem \ref{teo:ArbitraryDistanceBetweenSignalEdges} and an identification operation between
two edges of the same graph (see Definition \ref{de:ident}), an idea already used in \cite{Borges:boletin}.


\section{Minimal Senders}\label{sec:MinimalSenders}


A sender is {\em minimal} if none of its proper subgraphs is a sender with the same signal edges $e,f$.

Every pair of graphs $(G,H)$ from $\Gamma_2'$ has minimal senders.
\begin{lema}\label{le:minimalidad}
If $G$ and $H$
have a negative (resp. positive) $(G,H,e,f)$-sender,
then they have a minimal negative (positive) $(G,H,e,f)$-sender.
\end{lema}

\begin{proof}
	Let $F$ be a negative (positive) $(G,H)$-sender with
	signal edges $e,f$.
	
By Hypothesis the set
\[
S=\setdef{k\in\mathbb N}{F'\subseteq F \text{ is a negative (positive) $(G,H,e,f)$-sender with $k$ edges}}
\]
is non empty.
Let $m$ be the least element of $S$.
There is a graph $F'\subseteq F$ with $m$ edges which is a negative (positive) sender with signal edges $e$ and $f$.
Any proper subgraph $F''$ of $F'$ has less than $m$ edges and it can not be a negative (positive) sender with
the same signal edges, because $|E_{F''}|$ does not belong to $S$.
Therefore $F'$ is minimal.

\end{proof}

\begin{observacion}\label{ob:ColModulares}
Let $F=\tup{V_F,E_F}$ be a graph
with connected components $F_j=\tup{V_j,E_j}$ for $0\leq j<k$
and $G,H$ two connected graphs.

A coloring
$\col:E_F\longrightarrow \set{0,1}$ is $(G,H)$-good
iff
there is a $(G,H)$-good coloring $\col_j:E_j\longrightarrow \set{0,1}$
for each $0\leq j<k$ such that $\col_j=\col\upharpoonright_{E_j}$. 
\end{observacion}

\begin{teorema}\label{teo:MinimalSendersAreConnected}
If $G,H$ have negative or positive
senders, then every minimal $(G,H,e,f)$-sender,  either negative
or positive, is connected.
\end{teorema}

\begin{proof}

Suppose the pair $(G,H)$ have negative senders.
Then there is a minimal negative
$(G,H,e,f)$-sender $F=\tup{V_F,E_F}$ by Lemma \ref{le:minimalidad}.

If $F$ is not connected, then we can assume that it has connected components
$F_0,F_1,\ldots, F_{k-1}$
with $F_i=\tup{V_i,E_i}$ for each $0\leq i<k$. 

Suppose $e\in E_j$ and $f\in E_k$ with $j\neq k$. 
Let $\col:E_F\longrightarrow \set{0,1}$ 
and $\col':E_F\longrightarrow \set{0,1}$ be good colorings such that
$\col(e)=\col'(f)$.

By Observation \ref{ob:ColModulares}, $\col\upharpoonright_{E_j}$ and $\col'\upharpoonright_{E_k}$  
are good colorings of components $F_j$ and $F_k$ respectively.
Again by Observation \ref{ob:ColModulares} the coloring defined by
\[
\tilde \col(x)=
		\left\{
			\begin{matrix}
			\col(x) & \text{if $x\in E_i\;\;\; i\not=k$}\\
			\col'(x) & \text{if $x\in E_k$ }
			\end{matrix}
		\right.
\]
is good, and $\tilde \col(e)=\tilde \col(f)$,
which contradicts the fact that $F$ is a negative sender.
Therefore, signal edges $e$ and $f$ must belong to the same connected component $F_j$ of $F$. 

Now if $\col:E_F\longrightarrow \set{0,1}$ is a good coloring its restriction
$\col\upharpoonright_{E_j}:E_j\longrightarrow \set{0,1}$ is a good coloring of $F_j$. 
Then we can conclude that $F_j$ is a connected $(G,H,e,f)$-negative sender,
but this shows that $F$ can not be minimal.

Therefore any minimal  $(G,H,e,f)$-negative sender is necessarily connected.

The argumentation for positive senders is analogous.
\end{proof}

\begin{teorema}\label{teo:ArbitraryDistanceBetweenSignalEdges}
	For every natural number $n$ there is a minimal negative (positive) $(G,H)$-sender 
	such that the distance between its signal edges is at least $n$. 
\end{teorema}

\begin{proof}
	Suppose $F$ is a negative $(G,H)$-sender with signal edges $e$ and $f$ and such that
	the distance between them is $n$. Such a sender exists by
	Proposition \ref{prop:BurrNegativeSenders}.
	Hence the set
	\[
	S=\setdef{k\in\mathbb N}{F'\subseteq F \text{ is a negative (positive) $(G,H,e,f)$-sender with $k$ edges}}
	\]
	is non empty.
	Let $m$ be the least element of $S$.
	There is a negative (resp. positive) sender $F'$ with $m$ edges and signal edges $e,f$ separated by a distance $n$.
	By an argument similar to the one given in the proof of Lemma \ref{le:minimalidad} this $F'$
	is a minimal negative (positive) sender. The distance between $e$ and $f$ in $F'$ must be at least
	the distance between them in $F$ because $F'$ is a subgraph of $F$.

\end{proof}



\section{Minimal Ramsey graphs with arbitrarily large cycles}\label{MinimalSendersWithArbitrarilyLargeCycles}

\begin{definicion}\label{de:ident}
	Let $G=(V_G,E_G)$ be a graph and let $x=(a,b), x'=(c,d)$ be two edges of $G$.
	
	We define a new graph $G'=G[x\ident x']$ as follows:
	\[
	V_{G'}=V_G\setminus\set{a,b}
	\]
	
	\begin{align*}
		E_{G'}	 =& \setdef{(u,v)\in V_{G'}^2}{(u,v)\in E_G}
				\cup\setdef{(u,c)\in V_{G'}^2}{(u,a)\in E_G}\\
			&	\cup\setdef{(u,d)\in V_{G'}^2}{(u,b)\in E_G}
	\end{align*}
\end{definicion}
We say that $F$ is obtained from $G$ by {\em identification} of $x$ and $x'$.
Notice that $x$ is replaced by $x'$ i.e. $x$ is not an edge of $F'$
but $x'$ is.

\begin{lema}\label{le:NoGoodColoring}
	If $F$ is a negative $(G,H,e,f)$-sender. 
	Then $F[e\ident f]$ have no good colorings.
\end{lema}

\begin{proof}
	Suppose $F'=F[e\ident f]$ has a good coloring $\col'$.
	Now consider the coloring $\col$ on the eges of $F$:
	\[
	\col(x)=
	\begin{cases}
		\col'(x)	&	\quad\text{if $x\neq e$ and $x\neq f$}\\
		\col'(f)	& \quad\text{if $x= e$ or $x= f$}
	\end{cases}
	\]
	The coloring $\col$ can not be good since $\col(e)=\col(f)$ and $F$ is a negative sender.
	Thus $F$ contains either a red copy of $G$ or a blue copy of $H$
	but then $\col'$ is not a good coloring for $F'$
	because any copy of $G$ or $H$ contained in $F$ is contained also in $F'$
	and $\col(x)=\col'(x)$ for every $x\in E_{F'}\setminus f$.
\end{proof}

\begin{teorema}\label{Teo:CicloNoAcotado}
	For any natural number $n$ there is a graph in $\mathcal R(G,H)$ 
	containing a cycle of length at least $n$.
\end{teorema}

\begin{proof}
	Let $F$ be a minimal negative $(G,H)$-sender with signal edges $e=(a,b)$ and $f=(c,d)$
	such that $d(e,f)\geq n$.
	Such a negative sender exists because of Theorem \ref{teo:ArbitraryDistanceBetweenSignalEdges}.
	Hence there are vertices $u\in e, v\in f$
	such that $m=d(u,v)\geq n$.
	 
	Suppose $P$ is a path between $u$ and $v$ with length $m$.
	Thus
	\(
	P=v_0,v_1,\ldots, v_m
	\)
	with $u=v_0$ and $v=v_m$.
	This path exists because $F$ is connected by Theorem \ref{teo:MinimalSendersAreConnected}.
	
	If $u$ and $v$ are identified in $F'=F[e\ident f]$ 
	it is immediate that the path $P$ corresponds to a cycle
	of lenght $m$.
	
	If $u$ and $v$ are not identified in $F'$ then either $u=a$ and $v=d$
	or either $u=b$ and $v=c$. Suppose $u=a$ and $v=d$ (the other case is analogous).
	Then $P$ corresponds to the path $P'$ in $F'$:
	\[
	P'=u_0,u_1,\ldots, u_m
	\]
	where $u_0=c$, $u_m=d$ and $u_i=v_i$ if $i\neq 0$ and $i\neq m$.
	As $c$ and $d$ form the edge $f$, the path
	\[
	u_0,u_1,\ldots, u_m,c
	\]
	is a cycle of length $m+1>n$.
	
	We have proved so far that $F'$ contains a cycle of length at least $n$.
	It remains to prove that $F'\in\mathcal R(G,H)$.
	First, note that $F$ is a negative sender, thus
	$F'$ have no good colorings by Lemma \ref{le:NoGoodColoring}.
	Since $F$ is a negative sender, it has good colorings.
	If $x\notin\set{e,f}$ is an edge of $F$, then $F\setminus x$ still has good colorings
	but $F\setminus x$ is not a negative sender because of the minimality of $F$
	hence there is a good coloring $\col$ of $F\setminus x$ with $\col(e)=\col(f)$.
	This induces a good coloring $\col'$ of $F'\setminus x$ given by $\col'(y)=\col(y)$
	for every edge $y$ of $F'$.
	
	If $x=e$ or $x=f$ then $F'\setminus x=F'\setminus f$.
	As $f$ was the ``conflictive'' edge, we have that $F'\setminus f$ has good colorings.
	
	Hence $F'\in\mathcal R(G,H)$ as we wanted to prove.
\end{proof}

\section{Conclusions, observations and Further Research}\label{sec:FurtherResearch}

Theorem \ref{teo:MinimalSendersAreConnected} says that for any given pair of graphs $G$ and $H$,
if they have negative (positive) senders, then any minimal negative (positive) $(G,H)$-sender 
must be connected.

We extend Proposition \ref{prop:BurrNegativeSenders} 
to minimal senders in Teorem \ref{teo:ArbitraryDistanceBetweenSignalEdges}.
On the other hand Theorem \ref{Teo:CicloNoAcotado} says we can have
graphs in $\mathcal R(G,H)$ with arbitrarily large diameters.

The class $\Gamma_2'$ contains all graphs that do not get disconnected
by the removal of an edge. 
Proposition \ref{prop:BurrNegativeSenders} is still valid when $G$ and $H$ belong to $\Gamma_2'$ 
and the Arrowing relation is replaced by the strong Arrowing relation \cite{Burr}.
Hence our results are still valid for the strong Arrowing if $G$ and $H$ 
belong to $\Gamma_2'$.



\section*{Bibliography}
\begin{thebibliography}{99}
	
	\bibitem{Borges:boletin} N. Borges. A sufficient condition for first order non-definability of arrowing problems, {\em Bolet\'in de la Asociaci\'on Matem\'atica Venezolana}, Volume XX, Issue 2, 2013, pp. 109 -- 118.
	
	
	\bibitem{Burr} S. Burr and J. Ne\v{s}et\v{r}il and V. R\"{o}dl. On the use of senders in generalized Ramsey theory for graphs.
	{\em Discrete Mathematics}, Volume 54, 1985, pp. 1 -- 13.
	
	\bibitem{Burr-MoR} S. Burr. On the Computational Complexity of Ramsey-Type Problems. {\em Mathematics of the Ramsey Theory}.
	Springer, 1990.
	
	\bibitem{diestel2005graph} R. Diestel. {Graph theory (Graduate texts in mathematics)}. Springer, 2005.
	
	\bibitem{graham1990ramsey} {R. Graham, B. Rothschild and J. Spencer}. {\em Ramsey theory}. {John Wiley \& Sons}, 1990.  
	
	
	
	\bibitem{lange2012use} A. Lange, S. Radziszowski, and X. Xu. Use of MAX-CUT for Ramsey arrowing of triangles.
	arXiv preprint arXiv:1207.3750, 2012.
	
	
	\bibitem{Rodl-Sieggers}
	{V. R\"{o}dl and M. Sieggers}
	{On Ramsey Minimal Graphs}.
	{\em SIAM J. Discrete Math.}
	{Volume 22(2)},
	{2008},
	{pp. 467 -- 488}
	
	
	\bibitem{Rutenburg} V. Rutenburg. Complexity of generalized graph coloring. {\em Mathematical Foundations of Computer Science, 12th Symposium. Bratislava, Czeckoslovakia, 25-29 de agosto de 1986.} Springer. Volume {233 LNCS} pp. 573--581.  
	
	
	
	\bibitem{Schaefer} M. Schaefer. {Graph Ramsey theory and the polynomial hierarchy}. 
	{\em Proceedings of the 31st Annual ACM Symposium on Theory of Computing}. Volume 1999 pp. 592--601.
	
	
	
	
	\bibitem{schelp2012some} {Schelp, Richard H}. Some Ramsey--Tur{\'a}n type problems and related questions.
	{\em Discrete Mathematics},
	{Volume 312(14)}, 2012, pp. 2158--2161.
	
	
	
	
	
	
	
	
	
	
	
	
	
	
	
	
	
	
	
	
	
\end{thebibliography}
\end{document}